\documentclass[11pt]{article}
\usepackage[a4paper,top=2.2cm, bottom=3cm]{geometry}
\usepackage{amsmath}
\usepackage{amsfonts}
\usepackage{amsthm}
\usepackage{graphicx}
\usepackage[utf8]{inputenc}
\usepackage{hyperref}
\newtheorem{lemma}{Lemma}
\newtheorem{theorem}{Theorem}
\newtheorem{proposition}{Proposition}
\newcommand{\Paths}{\mathcal{P}}
\newcommand\numpath[2]{Z^{#1}_{#2}}
\newcommand\oeis[1]{\cite[\href{https://oeis.org/#1}{#1}]{oeis}}

\newcommand\s{s}

\newcommand\muco{\multicolumn2{c}}
\def\muco#1{{\setbox0=\hbox{#1}\dimen0=0,27\wd0
    \null\hskip-\dimen0\,\unhbox0\,
    \hskip-\dimen0\null
  }}

\title{The Number of Convex Polyominoes with Given Height and Width}
\author{Kevin Buchin, Man-Kwun Chiu, Stefan Felsner,
Günter Rote, André Schulz}

\begin{document}
\maketitle
\begin{abstract}
  We give a new combinatorial proof for the number of convex polyominoes whose
  minimum enclosing rectangle has given dimensions.
We also count the subclass of these polyominoes that contain the lower
left corner of the enclosing rectangle (\emph{directed} polyominoes).
We indicate how to sample random polyominoes in these classes.
As a side result, we calculate the first and second moments of the
number of common points of two monotone lattice paths between two
given points.
\end{abstract}

{\makeatletter{\renewcommand*{\@makefnmark}{}\makeatother
\footnotetext{This research was done 
at the
15th European Research Week on Geometric Graphs
(GGWeek 2018),
September 10--14, 2018,
in Haus Tornow am See (M\"arkische Schweiz, Germany).
The workshop was supported by the Deutsche Forschungsgemeinschaft (DFG)
through the Research Training Network
\emph{Facets of Complexity}
and
the collaborative DACH project
\emph{Arrangements and Drawings}.}

}

\tableofcontents








\section{Convex Polyominoes: Overview of Results}
\label{sec-convex}

A \emph{polyomino} is a connected set of lattice squares, where two
squares are considered as adjacent if they share an edge.
A polyomino is \emph{convex} if the intersection with every horizontal
or vertical line is a connected interval,
see Figure~\ref{fig:example} for an example.
The \emph{enclosing rectangle} $R$ is the smallest axis-parallel
rectangle containing the polyomino.
\begin{figure}
  \centering
  \includegraphics{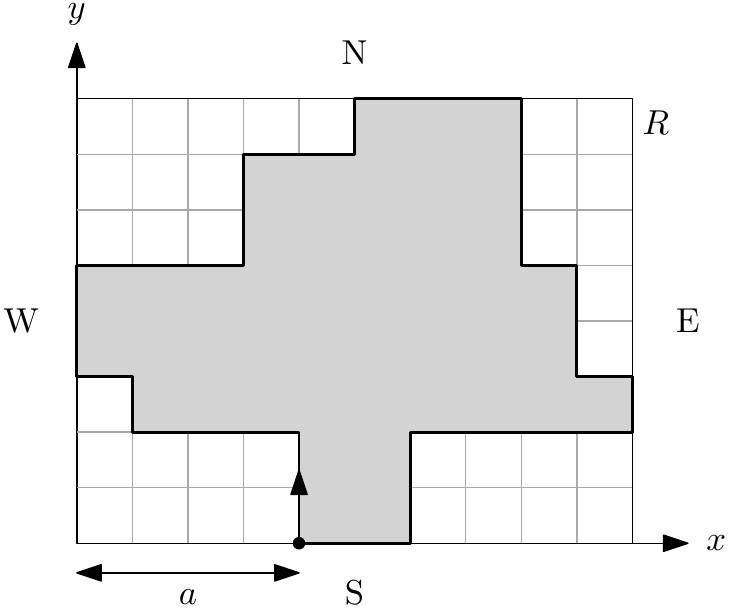}
  \caption{A convex polyomino of width $w=10$ and height $h=8$ in its
    enclosing rectangle~$R$.
    We traverse the boundary clockwise, starting at the leftmost
    point $(a,0)$ on the bottom side.
  The distance $a$ of this point from the left edge will be used as a parameter.}
  \label{fig:example}
\end{figure}
The \emph{height} $h$ and \emph{width} $w$ of a polyomino are defined
as the corresponding dimensions of $R$.
The perimeter (length of the boundary) of a convex polyomino is
$2(h+w)$.
Throughout, 
we will denote the semi-perimeter by $\s := h+w$.

\begin{theorem}\label{th:all-convex}
  The number of convex polyominoes of width $w\ge1$ and height $h\ge1$
is
\begin{equation}
  \label{P}
    P_{wh}=\binom{2\s-4}{2w-2} + \frac{2\s-5}2\binom{2\s-6}{2w-3}
      -2      (\s-3)\binom{\s-2}{w-1}\binom{\s-4}{w-2},
\end{equation}
where $\s=w+h$,
see Table~\ref{tab:Pwh}.
\end{theorem}
This formula was first proved by Gessel~\cite{Ges}, who derived it from the generating
function
\begin{equation}\label{GF}
  \sum_{h=1}^\infty
\sum_{w=1}^\infty
P_{wh}x^{2w}y^{2h} = 
 \frac
{x^2y^2\,B(x^2,y^2)}
  {\Delta(x,y)^2} - \frac{  4x^4y^4}{\Delta(x,y)^{3/2}} 
\end{equation}
with
\begin{align*}
  B(x,y) &=     1-3x-3y+3x^2+3y^2+5xy-x^3-y^3-x^2y-xy^2-xy(x-y)^2,\\
  \Delta(x,y) &=
                1- 2x^2 - 2y^2 + (x^2 - y^2)^2 
                      = (1 + x + y)(1 + x - y)(1 - x + y)(1- x- y),
\end{align*}
which is due to Lin and Chang~\cite[Eq.~\thetag{27}]{LC}.
The two terms in \eqref{GF} correspond to
the
positive and the negative terms in~\eqref{P}.
 Another proof for the formula~\eqref{P} was given by Guo and Zeng~\cite{GZ}.
They exploit a correspondence between convex polyominoes and noncrossing
 lattice paths, similar as in our proof.
 They use generating functions
 of Jacobi polynomials in order to evaluate the expressions that arise.
We will review this proof in Section~\ref{sec:GZ}.

We present a new combinatorial proof of formula~\eqref{P} in
Section~\ref{sec:general}, and we assign
some meaning
to
its
 terms. 
Based on this interpretation,
in Section~\ref{sec:sampling}, we discuss an
 acceptance-rejection method for sampling from this set of convex polyominoes,
  with an acceptance probability tending to one for large parameters
  $w$ and~$h$.

The summation of $P_{wh}$ 
over all pairs $w,h$ with fixed sum $\s$
gives
the number of convex polyominoes with perimeter $2\s$, which we
denote by $P_\s$. The resulting formula, 
\begin{align} \nonumber
  P_\s
=\sum_{w+h=\s} P_{wh}
  &=
      4^{\s-4}(2\s+3)
 -
4
(2\s-7)\binom{2\s-8}{\s-4}
\\  &= \label{Pm}
      4^{\s-4}(2\s+3)
 -
(2\s-6)\binom{2\s-6}{\s-3}
,\ \text{ for $\s\ge 4$}
\end{align}
has been known before \cite{DV,Kim}. (The reader should beware that
the literature contains
formulas 
in terms of various different quantities
instead of the parameter $\s$ that we have chosen.)
\begin{table}
  \centering
\begin{tabular}[c]{rc@{}c@{}c@{}c@{}c@{}c@{}c@{}c@{}c@{}c@{}c@{}c@{}c@{}c@{}c@{}c@{}c@{}c@{}cr}
$\s$&\multicolumn{19}{c}{$P_{wh}$}&$P_\s$\\\hline
$2 $&&&&&&&&&& \muco{1} &&&&&&&&&& 1 \\
$3 $&&&&&&&&& \muco{1}&&\muco{1} &&&&&&&&& 2 \\
$4 $&&&&&&&& \muco{1}&&\muco{5}&&\muco{1} &&&&&&&& 7 \\
$5 $&&&&&&& \muco{1}&&\muco{13}&&\muco{13}&&\muco{1} &&&&&&& 28 \\
$6 $&&&&&& \muco{1}&&\muco{25}&&\muco{68}&&\muco{25}&&\muco{1} &&&&&& 120 \\
$7 $&&&&& \muco{1}&&\muco{41}&&\muco{222}&&\muco{222}&&\muco{41}&&\muco{1} &&&&& 528 \\
$8 $&&&& \muco{1}&&\muco{61}&&\muco{555}&&\muco{1110}&&\muco{555}&&\muco{61}&&\muco{1} &&&& 2344 \\
$9 $&&& \muco{1}&&\muco{85}&&\muco{1171}&&\muco{3951}&&\muco{3951}&&\muco{1171}&&\muco{85}&&\muco{1} &&& 10416 \\
$10 $&& \muco{1}&&\muco{113}&&\muco{2198}&&\muco{11263}&&\muco{19010}&&\muco{11263}&&\muco{2198}&&\muco{113}&&\muco{1} && 46160 \\
$11 $& \muco{1}&&\muco{145}&&\muco{3788}&&\muco{27468}&&\muco{70438}&&\muco{70438}&&\muco{27468}&&\muco{3788}&&\muco{145}&&\muco{1} & 203680 \\
\end{tabular}
  \caption{The number $P_{wh}$ of convex polyominoes of width $w$ and
    height $h$ for $2\le \s=w+h\le 11$ \oeis{A093118}.
    The width $w$ increases from
$1$ to $\s-1$ in row $\s$,
while $h$ decreases and $\s=w+h$ remains constant. The row sum
    $P_\s$ is given in the rightmost column. It is the number of convex polyominoes
    of perimeter $2\s$ \oeis{A005436}.}
  \label{tab:Pwh}
\end{table}

A \emph{directed convex polyomino} is a convex polyomino that contains
the lower left corner of its enclosing rectangle $R$.

\begin{theorem}\label{directed}
  The directed convex polyominoes of width $w\ge1$ and height $h\ge1$
are counted by the squared binomial coefficients:
\begin{equation}\label{formula-directed}
D_{wh}
  =\binom{\s-2}{w-1}^2,
\end{equation}
where $\s=w+h$,
see Table~\ref{tab:directed}.
\end{theorem}
This formula was proved 
by
Barcucci, Frosini, and Rinaldi
~\cite{BFR}. 
  It 
  is implicit also in
Guo and Zeng~\cite
{GZ}, see
Section~\ref{sec:GZ}.
We will give a different proof in Section~\ref{sec:directed-width}.
The corresponding  generating function
\begin{displaymath}
    \sum_{h=1}^\infty
    \sum_{w=1}^\infty
D_{wh}x^{2w}y^{2h} = x^2y^2\big/\sqrt{\Delta(x,y)}
\end{displaymath}
has been known as well, see
Lin and Chang~\cite[Eq.~\thetag{25}]{LC}.

The formula~\eqref{formula-directed} suggests that there should be a bijection between these
polyominoes and pairs of objects that are counted by
the binomial coefficients.
Indeed, from our proof of Theorem~\ref{directed}, we can derive a
bijection
to pairs of monotone lattice paths. This is described in
Section~\ref{bijection}.
Such a bijection can be used to
generate a random sample from the
polyominoes from this class, since it is straightforward to
generate random monotone lattice paths.
The proof of Barcucci et~al.~\cite{BFR} mentioned above is also
bijective:
They used a bijection between
directed polyominoes and so-called \emph{two-colored Grand-Motzkin paths}.
It is easy to see that
{two-colored Grand-Motzkin paths} are in one-to-one correspondence
with
pairs of lattice paths, but
the bijection to polyominoes is different from ours.
We will say more 
in
Section~\ref{sec:motzkin}.

As above, we can sum the rows of the table and
conclude that  the number of directed convex polyominoes with
perimeter $2\s$
are the central binomial coefficients
\begin{equation*}
D_\s =   \binom{2\s-4}{\s-2},
\end{equation*}
  which was also known before \cite{LC,B-M},
\oeis{A000984}.
\begin{table}
  \centering
\begin{tabular}[c]{rc@{}c@{}c@{}c@{}c@{}c@{}c@{}c@{}c@{}c@{}c@{}c@{}c@{}c@{}c@{}c@{}c@{}c@{}cr}
$\s$&\multicolumn{19}{c}{$D_{wh}$}&$D_\s$\\\hline
$ 2 $&&&&&&&&&& \muco{1} &&&&&&&&&& 1 \\
$ 3 $&&&&&&&&& \muco{1}&&\muco{1} &&&&&&&&& 2 \\
$ 4 $&&&&&&&& \muco{1}&&\muco{4}&&\muco{1} &&&&&&&& 6 \\
$ 5 $&&&&&&& \muco{1}&&\muco{9}&&\muco{9}&&\muco{1} &&&&&&& 20 \\
$ 6 $&&&&&& \muco{1}&&\muco{16}&&\muco{36}&&\muco{16}&&\muco{1} &&&&&& 70 \\
$ 7 $&&&&& \muco{1}&&\muco{25}&&\muco{100}&&\muco{100}&&\muco{25}&&\muco{1} &&&&& 252 \\
$ 8 $&&&& \muco{1}&&\muco{36}&&\muco{225}&&\muco{400}&&\muco{225}&&\muco{36}&&\muco{1} &&&& 924 \\
$ 9 $&&& \muco{1}&&\muco{49}&&\muco{441}&&\muco{1225}&&\muco{1225}&&\muco{441}&&\muco{49}&&\muco{1} &&& 3432 \\
$ 10 $&& \muco{1}&&\muco{64}&&\muco{784}&&\muco{3136}&&\muco{4900}&&\muco{3136}&&\muco{784}&&\muco{64}&&\muco{1} && 12870 \\
$ 11 $& \muco{1}&&\muco{81}&&\muco{1296}&&\muco{7056}&&\muco{15876}&&\muco{15876}&&\muco{7056}&&\muco{1296}&&\muco{81}&&\muco{1} & 48620 \\
\end{tabular}
  \caption{The number $D_{wh}$ of directed convex polyominoes of width $w$ and
    height $h$. These are the squared binomial coefficients~\oeis{A008459}.}
  \label{tab:directed}
\end{table}

Our investigation of pairs of lattice paths has led, as a side result,
to formulas for the
first and second moments of the number of intersections between two
lattice paths, see Section~\ref{moments}.

For completeness, we state the well-known formulas for \emph{parallelogram
polyominoes}, which contain both the lower left and the upper right corner of~$R$.
A proof can be easily obtained by Lemma~\ref{lemma-intersecting} below
(see Section~\ref{sec:directed-width}).

\begin{proposition}\label{proposition-parallelogram}
  The number of\ parallelogram polyominoes of width $w\ge1$ and height $h\ge1$
is given by a Narayana number
\begin{equation*}
  B_{wh} = 
  {\frac{1}{\s-1}\binom{\s-1}w\binom{\s-1}h}
.\qquad\textup{\oeis{A001263}}
\end{equation*}
  The number of convex parallelogram polyominoes of perimeter $2\s$
  is equal to the $(\s-1)$th Catalan number
\begin{equation*}
  B_{\s} = \frac{1}{\s}
  { \binom{2\s-2}{\s-1}}. \qquad\textup{\oeis{A000108}}
  \eqno\qed
\end{equation*}
\end{proposition}
For a bijective proof of the second formula and
some references to its history see Stanley~\cite[p.~66]{stanley}.
\section{Conventions}

Before going into the proofs, we clarify some notations that we will use.
We label the four sides of $R$ by N, S, E, W. Accordingly, we call the corners of
$R$ NE corner, SE corner, and so on.
We place $R$ in the first quadrant with the SW corner at the origin.



By
an \emph{edge}, we mean a unit-length boundary segment of a grid
square.
On the boundary of a polyomino,
several edges may lie straight in succession,
but they are considered as separate edges.

If $P$ is a polyomino with enclosing rectangle $R$ and
$K\in\{\textrm{N, S, E, W}\}$, then the \emph{$K$-side} of $P$ is the
intersection of $P$ with the respective side of $R$.



\section{Directed Convex Polyominoes}
\label{sec:directed-width}
We start with
the easier and more restricted case of directed polyominoes, i.e.,
polyominoes that contain the SW corner $(0,0)$
(Theorem~\ref{directed}). This serves also as a preparation for
our treatment of general polyominoes in Section~\ref{sec:general},
where some of the arguments will recur.


We will relate directed polyominoes to certain
monotone paths in the square lattice.
A \emph{monotone path}
is a lattice path whose steps go in only two orthogonal directions.

Traverse the boundary clockwise, starting from the origin. Between
the N-side and the E-side, the boundary intersects the diagonal
 $y-x=h-w$ through the NE corner at a unique point
 $X=(w,h)-(i,i)$ for some $i\ge 0$.
see Figure~\ref{fig:split}.

\begin{figure}
  \centering
  \includegraphics{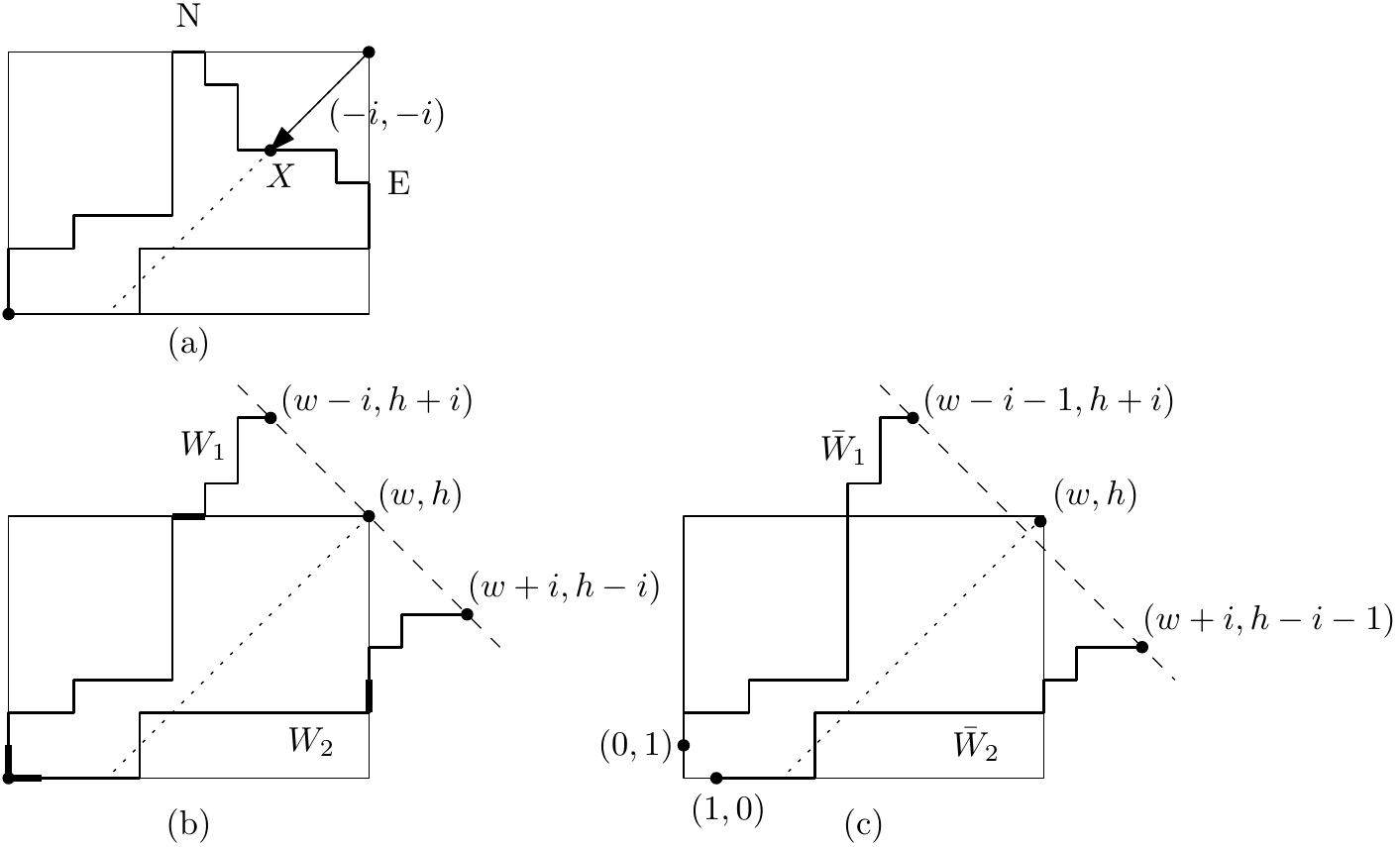}
  \caption{Splitting the boundary at the diagonal (a) and reflecting at the N-side
    and E-side. (b)~An implicit edge on each side of $R$ is highlighted. (c) The implicit edges from the N-side and the E-side
    are removed.}
  \label{fig:split}
\end{figure}

We cut the boundary path at $X$.
We reflect the path between the N-side and $X$ across the N-side
of $R$.
Similarly, we reflect the path between $X$ and the E-side across the E-side
of $R$. We get two monotone paths $W_1$ and $W_2$,
ending in $(w,h)+(-i,i)$,
and $(w,h)+(i,-i)$.
It would be easy to count pairs of these paths but the transformation
would not be bijective. We have to ensure two properties.
\begin{enumerate}
\item [(a)] The boundary must not cross itself.
\item [(b)] The boundary must have at least one edge
  on each of the four sides of $R$.
\end{enumerate}
We take care of (b) by declaring the first edge on the E- and the N-side
to be an \emph{implicit edge}. If we remove these edges from the
two paths, we obtain
paths $\bar W_1$ and $\bar W_2$ that end in
 $(w,h)+(-i-1,i)$,
 and $(w,h)+(i,-i-1)$.
The implicit edges on the W-side and the S-side are easier to take care of: We
know that the boundary starts with an up-step and ends with a
left-step.
Thus, we simply let
$\bar W_1$ start in the point $(0,1)$,
and we let
$\bar W_2$ start $(1,0)$.
Now we have a bijection: There is a one-to-one correspondence
between directed polyominoes and pairs
$(\bar W_1,\bar W_2)$ of \emph{non-intersecting} monotone lattice paths
from $(0,1)$ to $(w,h)+(-i-1,i)$,
and from
 $(1,0)$ to $(w,h)+(i,-i-1)$, for some $0\le i<\min\{w,h\}$.
Indeed, given such a
 pair 
 of {non-intersecting} paths, we re-insert the implicit edge into each
 path, reflect the paths outside $R$ back, and join the two pieces
 together.
 This process is reversible, and it does not introduce or destroy intersections.

We still have to take care of self-intersections.
It is known how to count
 pairs of non-intersecting monotone lattice paths. 
 We will use the basic Lemma~\ref{lemma-intersecting} below,
 which is the first nontrivial case of
 the Gessel--Viennot Lemma.
We include its proof because we will extensively use the underlying uncrossing idea.
The Gessel--Viennot Lemma is more general: it counts families of non-intersecting paths between more
 than two pairs of terminal nodes in acyclic directed graphs.


 For integer points $p,q,u,v$ in the plane, we denote by
 $\bar X_{p,q}^{u,v}$
 the number of \emph{non-intersecting} pairs of monotone grid paths in the NE direction from $p$ to
 $u$ and from $q$ to $v$,
and by
 $ X_{p,q}^{u,v}$
 the number of \emph{intersecting}  pairs of such paths.
By
 $N_p^u$, we denote the number of monotone paths in the NE direction from $p$ to
 $u$.
This 
is equal to a binomial coefficient, for which we temporarily introduce
the abbreviation
\begin{equation*}
  Z_x^y=\binom {x+y}x.
\end{equation*}
Thus,
\begin{equation*}
  N_{(a,b)}^{(c,d)}
=  N_{(0,0)}^{(c-a,d-b)} = Z_{c-a}^{d-b}.
\end{equation*}
\begin{lemma}\label{lemma-intersecting}

Assume that the number of non-intersecting pairs of paths
 from $p$ to $v$ and from
$q$ to $u$ is zero.
Then the number of intersecting pairs of paths
 from $p$ to $u$ and from
$q$ to $v$ is

  \begin{equation}
    \label{eq:intersect}
 X_{p,q}^{u,v} = N_p^vN_q^u.
  \end{equation}
Consequently, the number of non-intersecting pairs of paths is
  \begin{equation}
    \label{eq:non}
    \bar X_{p,q}^{u,v} = N_p^uN_q^v - N_p^vN_q^u.
  \end{equation}
%
\end{lemma}
\begin{proof}
  For each intersecting pair of paths we can perform a crossover switch at
  the first intersection point, leading to a pair of paths where the
  end terminals are swapped.
Conversely, every pair of paths with swapped endpoints must intersect,
by assumption, and hence we can perform a crossover at the first
intersection point.
These two operations are inverses, thus establishing
\eqref{eq:intersect} by bijection.
\end{proof}

%


The lemma gives
\begin{align}
  D_{wh}&= \sum_{i\ge 0} \nonumber
          \bar X_{(0,1),(1,0)}^{(w-i-1,h+i),(w+i,h-i-1)}
\\&= \sum_{i\ge 0}\nonumber
 (   \numpath{w-i-1}{h+i-1}
    \cdot
    \numpath{w+i-1}{h-i-1}
    -
    \numpath{w-i-2}{h+i}
    \cdot
    \numpath{w+i}{h-i-2}
) \\
&=    \numpath{w-1}{h-1}
    \numpath{w-1}{h-1}
    -
    \numpath{w-2}{h}
    \numpath{w}{h-2}
+
    \numpath{w-2}{h}
    \numpath{w}{h-2}
    -
    \numpath{w-3}{h+1}
    \numpath{w+1}{h-3}
+
    \numpath{w-3}{h+1}
    \numpath{w+1}{h-3}
  - \cdots.
  \label{telescope}
\end{align}
This is a telescoping sum. Its terms become eventually zero, because
the horizontal distances ${w-i-1}$ and the vertical distances
${h-i-1}$ become negative for large~$i$.
Thus, only the very first term remains, and therefore
\begin{equation*}
  D_{wh}=    \numpath{w-1}{h-1}
    \numpath{w-1}{h-1}
          =\binom{\s-2}{w-1}^2.
       \eqno   \qed
\end{equation*}

\subsection{A bijection for directed convex polyominoes}
\label{bijection}

The formula $\binom{\s-2}{w-1}^2$ has an obvious combinatorial
interpretation: it counts ordered \emph{pairs} of monotone paths from
$(0,0)$ to $(w-1,h-1)$.  By analyzing the cancellation that happens in
the telescoping sum~\eqref{telescope}, we can derive a bijection to
directed convex polyominoes of height $h$ and width $w$.

\begin{figure}
  \centering
  \includegraphics{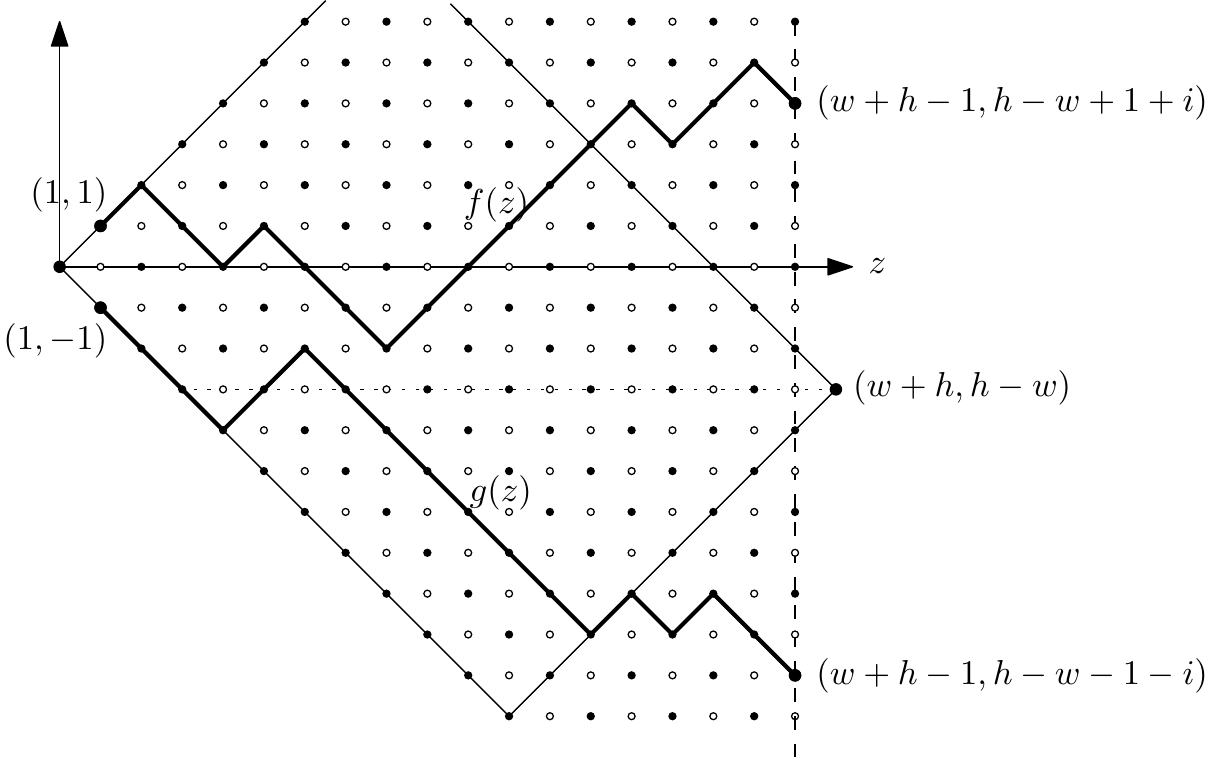}
  \caption{Two functions $f$ and $g$.}
  \label{fig:updown}
\end{figure}

We start from a pair of paths from $(0,0)$ to $(w-1,h-1)$.
It is more convenient to turn the grid by $45^\circ$ and scale it by
$\sqrt2$, using the transformation $\binom xy \mapsto \binom
{x+y}{y-x}$. We let the first path start at $(1,1)$ and the second path at
$(1,-1)$, like the desired paths of Figure~\ref{fig:split}c,
see Figure~\ref{fig:updown}.
This leads to two paths with endpoints
$
(w+h-1,h-w+1)$
and
$
(w+h-1,h-w-1)$.
The paths can be viewed as graphs of two functions
$f,g\colon \{1,2,\ldots,s-1\} \to \mathbb Z$ that make
steps $\pm 1$: $|f(z+1)-f(z)|=1$
and $|g(z+1)-g(z)|=1$.
Now we will modify these two functions to ensure that they don't intersect.
In general, we will have two functions on 
$\{1,2,\ldots,s-1\}$.
The starting heights are $+1$ and $-1$, and the
ending heights are given by
$
 {
h-w+1+2i}$ and
${
  h-w-1-2i}
$,
for some $i\ge 0$. Initially $i=0$.
If the paths intersect we make a crossover at the rightmost
intersection point, swapping the initial parts to the left of this
point.
Then
we shift the first path by $+2$ and the second path by $-2$.
Schematically, this is shown as follows (see also Figure~\ref{fig:bijection} for an example):
\begin{displaymath}
  \left[
    \begin{matrix}
      {+1\to h{-}w{+}1+2i}\\
{-1\to h{-}w{-}1-2i}
    \end{matrix}
  \right]
  \stackrel{\mathrm{crossover\strut}}\implies
  \left[
    \begin{matrix}
{-1\to h{-}w{+}1+2i}
\\{+1\to h{-}w{-}1-2i}
    \end{matrix}
  \right]
    \stackrel{\text{vertical shift\strut}}\implies
  \left[
    \begin{matrix}
 {+1\to h{-}w{+}1+2i+2}\\
{-1\to h{-}w{-}1-2i-2}
    \end{matrix}
  \right]
\end{displaymath}
We repeat this procedure as long as the functions intersect.
The same procedure would also work if we systematically take the
\emph{leftmost} intersection.
Both versions, if implemented naively, take quadratic time.
The current version is better because the parts to the right of the
intersection point are always moved further apart. This means that no
new intersection points appear there, and we can find the sequence of
rightmost intersection points in a single right-to-left sweep.

The changes left of the intersection points are best understood
in terms of their
 effect on the \emph{difference} function $\delta(z) := f(z)-g(z)$.
Swapping $f$ and $g$ amounts to negating~$\delta$, and the vertical
shift will add~$4$ to $\delta$. Together, these two operations
will perform the mapping $\delta\mapsto 2-\delta$.

\begin{figure}
\centering
	\begin{tabular}{cp{3mm}c}
	\includegraphics[page=1,scale=.9]{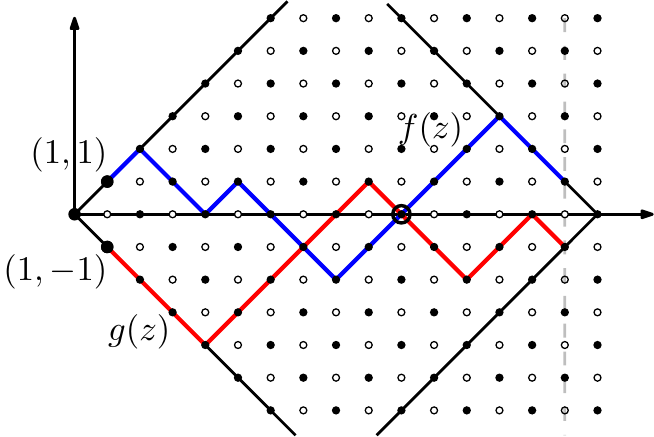} & &
		\includegraphics[page=2,scale=.9]{bijection} \\
		(a) && (b) \\[.5cm]
	\includegraphics[page=3,scale=.9]{bijection} & &
		\includegraphics[page=4,scale=.9]{bijection}\\
		(c) && (d)
	\end{tabular}
	\caption{(a)--(c): A sequence of untangling operations for an
          intersecting path pair. 
          The rightmost intersection pairs are marked.
          (d): The final mapping from the non-intersecting paths to the polyomino.}
	\label{fig:bijection}
\end{figure}

Performing this operation twice will restore the original function.
We can therefore leave the functions intact and, instead of searching
for zeros of $\delta$, alternately search for the values 0 and 4. The
following is
a streamlined version of this algorithm, where $\delta$ does not
explicitly appear.
After the sequence of intersection points is determined in the first phase, we compute
the resulting
noncrossing paths
$F,G\colon \{1,2,\ldots,s-1\} \to \mathbb Z$ of the bijection in a
left-to-right sweep.

\medskip
\begin{minipage}[t]{0.3\linewidth}
\begin{tabbing}
$--$ {Phase 1. Right-to-left
  sweep}:\\
$--$ {identify all intersection points}\\
$\textit{swapped} := \textit{False}$\\
$\textit{target} := 0$\\
$s := 0$\\
\textbf{for} $z := \s -1, \s -2, \ldots, 2, 1$:\\
\qquad \=\+\textbf{if} $f[z]-g[z]=\textit{target}:$\\
\qquad \=
$\textit{intersection}[z] := \textit{True}$\\
\>$\textit{swapped} := \textbf{not}\ \textit{swapped}$\\
\>$\textit{target} := 4-\textit{target}$\\
\textbf{else}:\\
\>$\textit{intersection}[z] := \textit{False}$
\end{tabbing}
\end{minipage}
\quad\vrule \quad\
\begin{minipage}[t]{0.4\linewidth}
\begin{tabbing}
$--$ {Phase 2. Left-to-right  sweep}:\\
\textbf{if} \textit{swapped}:\\
\qquad\=
$\textit{shift} := 2$\\
\textbf{else}:\\
\>$\textit{shift} := 0$\\
\textbf{for} $z := 1,2,\ldots, \s -1$:\+\\
\textbf{if} $\textit{swapped}:$\\
\qquad\=
$F[z] = g[z]+\textit{shift}$\\
\>$G[z] = f[z]-\textit{shift}$\\
\textbf{else}:\\
\>$F[z] = f[z]+\textit{shift}$\\
\>$G[z] = g[z]-\textit{shift}$\\
\textbf{if }$\textit{intersection}[z]:$\\
\>$\textit{swapped} := \textbf{not}\textit{ swapped}$\\
\>$\textit{shift} := \textit{shift} + 2$
\end{tabbing}
\end{minipage}

\subsection{Two-colored Grand-Motzkin paths}
\label{sec:motzkin}

When we found our proof,
we were not aware of the bijective proof
of 
Theorem~\ref{directed}
by
Barcucci, Frosini, and Rinaldi
~\cite{BFR}.
 There are differences as well as commonalities,
which we will now discuss,
 glossing over superficial details.
In particular, the idea of cutting at the point of intersection $X$ with the
diagonal through the upper right corner is also used.
This leads, in our case, to a pair of functions $f,g$ such that $f$
always remains above~$g$.
Barcucci et~al.~\cite[Section~3]{BFR}
represent the same information in terms of the difference function $\delta=f-g$.
This function
has only even values.
When $\delta$ makes a step of size~$2$, i.\,e.,
$\delta(z+1)=\delta(z)\pm2$, it is possible to infer the change of
$f$ and $g$.
However,
 if $\delta$ is stationary, this means that $f$ and $g$ either both
go up or both
go down.  One therefore has to distinguish these possibilities by
marking each horizontal step of $\delta$ with one of two colors.
After dividing $\delta$ by 2,
this leads  to the notion of a
``two-colored Grand-Motzkin path''
as an alternative way to represent the pair $(f,g)$:
A Grand-Motzkin path is a sequence of steps from $\{0,\pm 1\}$
starting and ending at the same height, 
and a Motzin path has the additional
requirement of never going below the initial height.

The main part of the proof concerns
the bijective
mapping from a pair $(f,g)$ of functions with difference
$\delta(\s -1)=2$ at the endpoint to a pair $(F,G)$ of
\emph{noncrossing} functions with endpoints at difference
$F(\s -1)-
G(\s -1)
=2+4i$, for some $i$.
In contrast to our bijection, which scans the sequence from right to left,
Barcucci et~al.~\cite[Section~3.1.3]{BFR} process the sequence from
left to right.
Their bijection has a
very elegant pictorial description
in terms of Grand-Motzkin paths
\cite[Figs.~13--14]{BFR}.
Translated to our notation, it can be described as follows:
Look for the leftmost point $z$ where
$\delta(z)=0$,
$\delta(z)=-2$,
$\delta(z)=-4$, etc.
Arriving at such a point, $f$ must have taken a down-step:
$f(z)=f(z-1)-1$, and  $g$ must have taken an up-step:
$g(z)=g(z-1)+1$.
 Modify $f$ by adding 2 to $f(z),f(z+1),\ldots$, changing the
down-step into an up-step, and modify $g$ similarly by
subtracting~$2$.
If the minimum value of $\delta$ is initially $2-2i$, then
this procedure is carried out cumulatively for  $i$
different positions $z$,
and the
final difference at the endpoints is
$
F(\s -1)-G(\s -1)
=2+4i$. We have achieved that $F(z)\ge G(z)+2$ throughout.
This bijection
is
clearly different
from
our algorithm, which was derived from the cancellation
in~\eqref{telescope} and from Lemma~\ref{lemma-intersecting}.
It is simpler 
as it does not need to
swap parts of $f$ and~$g$.

The inverse mapping, from $(F,G)$ to $(f,g)$ is just as easy:
If
$
F(\s -1)-G(\s -1)
=2+4i$,
perform a right-to-left scan:
look for the rightmost points $z$ where
$F(z)-G(z)=2i$,
$
2i-2$,
\ldots, 
$4$,
$
2$.
These are the points where up-shifts must be inverted to down-shifts
and vice versa.

\section{General Convex Polyominoes}
\label{sec:general}
\subsection{S-walks}
We will first introduce a definition that includes all convex polyominoes
but also allows some self-intersection of the boundaries.

An \emph{S-walk} is
a shortest closed walk on the grid that
spans its enclosing rectangle~$R$.
By \emph{spanning}, we mean that it contains at least one edge
on each of the four sides of $R$, while \emph{shortest} says that the length
of the walk equals the perimeter of $R$; in particular, the S-walk is
contained in $R$.
(The letter S in S-walks
has nothing to do with the \textsl{S}outh direction.)
We orient the S-walk by the convention that
the edges on the S-side are traversed from right to left, i.e.,  in the W direction.
(If the S-walk is non-intersecting, this means that it is oriented
clockwise.)
If $R$ is an $h\times w$ rectangle, then an S-walk consists of
exactly $h$ steps in each of the E and W directions,
and exactly $w$  steps in each of the N and S directions,
 so its total length is $2(h+w)$.
 An S-walk forms the boundary of a convex polyomino if and only if it
 does not intersect itself.

The S-walk will not necessarily visit the four sides in the correct
order.
We choose the leftmost point on the S-side as the starting point,
as shown in Figure~\ref{fig:example}.
Since we are moving
 in the W direction, the S-walk will meet the W-side
before the E-side. This leaves three possible orders in
which
the walk can hit the sides: SWNE, SNWE, and SWEN. They are
shown schematically in Figure~\ref{fig:walks}.
The ``normal'' order, which corresponds to the non-intersecting case, is
SWNE.
In the two other orders, an intersection is forced.
Removing the parts where the S-walk touches the rectangle boundary
decomposes it into four pieces,
each of which may be empty.
Each piece is a monotone grid path, and we classify these paths either
as \emph{rising}
or \emph{falling}, by looking at them as the graph of a function
and reading them from left to right (independent of the traversal
direction).
Thus, for example the path between the W-side and the N-side is
always rising.
\begin{figure}
  \centering
  \includegraphics{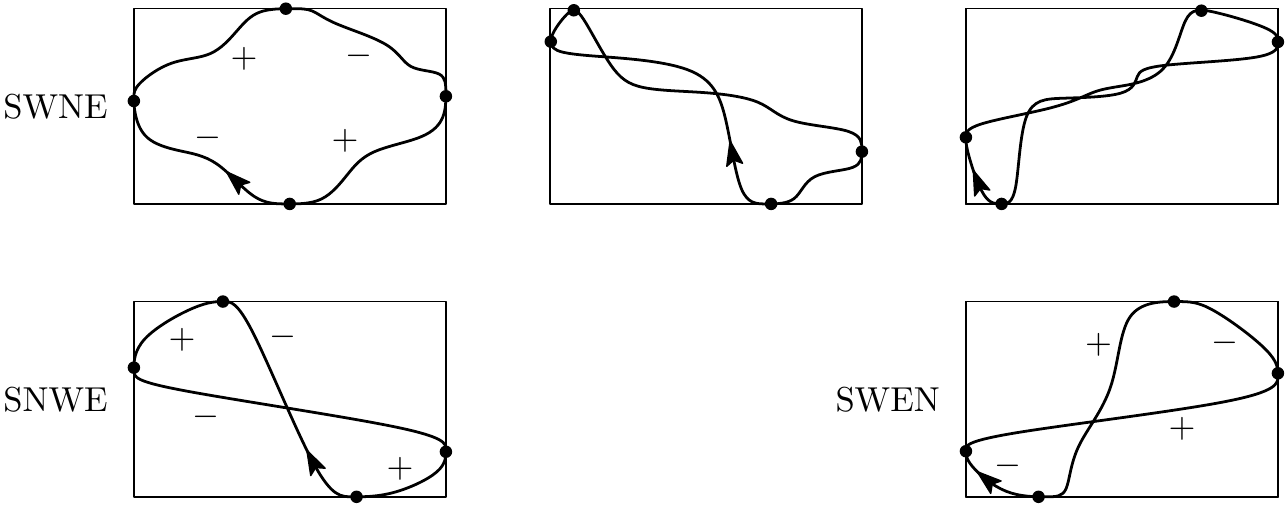}
  \caption{The possible orders in which an S-walk can touch the four
    sides of the enclosing rectangle, and the possibilities of
    self-intersections. Some rising and falling parts are marked by
    $+$ and~$-$.}
  \label{fig:walks}
\end{figure}
Whenever the walk hits a boundary edge, it alternates between
rising and falling. Thus, there will always be two rising
paths and two falling paths.  Looking at Figure~\ref{fig:walks}, it
is easy to convince oneself that a self-intersection can occur either
between the two rising paths, or between the two falling paths,
but not both simultaneously.  For example, in the SWNE case, the two falling paths
can cross only if the rightmost point of the walk on the N-side is to
the left of the leftmost point of the walk on the S-side, and the
bottommost point of the walk on the W-side is higher than the topmost
point of the walk on the E-side.
 This precludes the two rising paths from intersecting.
Of course, there may be multiple intersections between two rising/falling paths.

Thus, given an $h \times w$ enclosing rectangle, denoting by $\widetilde P_{wh}$ the number of all S-walks,
and by $P^*_{wh}$ the number of S-walks in which the \emph{rising} paths intersect,
we can compute
the number
$P_{wh}$
of convex polyominoes by the formula
\begin{equation}\label{all-convex}
  P_{wh} =
  \widetilde P_{wh} -
  2P^*_{wh}.
\end{equation}
The factor 2 reflects the fact that the intersection of the falling
paths is symmetric to the rising case.
In the following two subsections, we will determine $  \widetilde P_{wh}$ and
$P^*_{wh}$.

\subsection{The number of all S-walks}
\label{sec:number-S-walks}

We start the walk at the leftmost point $(a,0)$ on the S-side
with an up-step. Given the starting point in the range
$0\le a \le w-1$,
we code the walk 
by recording a sequence over the alphabet
$\{\mathrm V,\mathrm H\}$, representing the walking steps sequence whether we make a vertical or horizontal step.
It is not necessary to record the orientation in which the steps go
(left versus right, or up versus down),
because this information can be recovered by tracing the path when it hits the boundary of the rectangle.
Furthermore, the walk must contain at least one edge on each of the four sides of the rectangle. These four steps are \emph{implicit steps} and are not coded.
A code is translated into an S-walk as follows:
We start at $(a,0)$ and follow the $\{\mathrm V,\mathrm H\}$-code to
trace the walk.
Whenever we
touch one of the four sides W,N,S,E of $R$ for the first time, an
implicit step along that side is inserted.
In addition, the first up-step at $(a,0)$ is also not coded. If $a=0$, the
first up-step coincides with the implicit W-step.
(This corresponds to directed polyominoes.)
 Then the code
consists of
$2h-2$ V's
and
$2w-2$ H's.
If $1\le a\le w-1$, we need only
$2h-3$~V's.
In total, the number of S-walks is therefore
\begin{align}
  \widetilde P_{wh}
  &=
               \binom{2\s-4}{2w-2} + (w-1)\binom{2\s-5}{2w-2}
        =
               \binom{2\s-4}{2w-2} + \binom{2\s-6}{2w-3}\frac{2\s-5}2.
\label{s-walks}
\end{align}
The factor $w-1$ counts the number of nonzero possibilities for $a$.

This coding is very similar to the code used by
Hochstättler, Loebl and Moll~\cite{HLM}
for generating convex polyominoes of given perimeter,
see also Section~\ref{perimeter}.

\subsection{Self-intersecting S-walks}
\label{self-intersecting}

To count S-walks where the two rising paths intersect,
we adapt the approach of
Section~\ref{sec:directed-width},
splitting the walk at diagonals and moving falling paths outside
by reflection.
\begin{figure}
  \centering
  \includegraphics{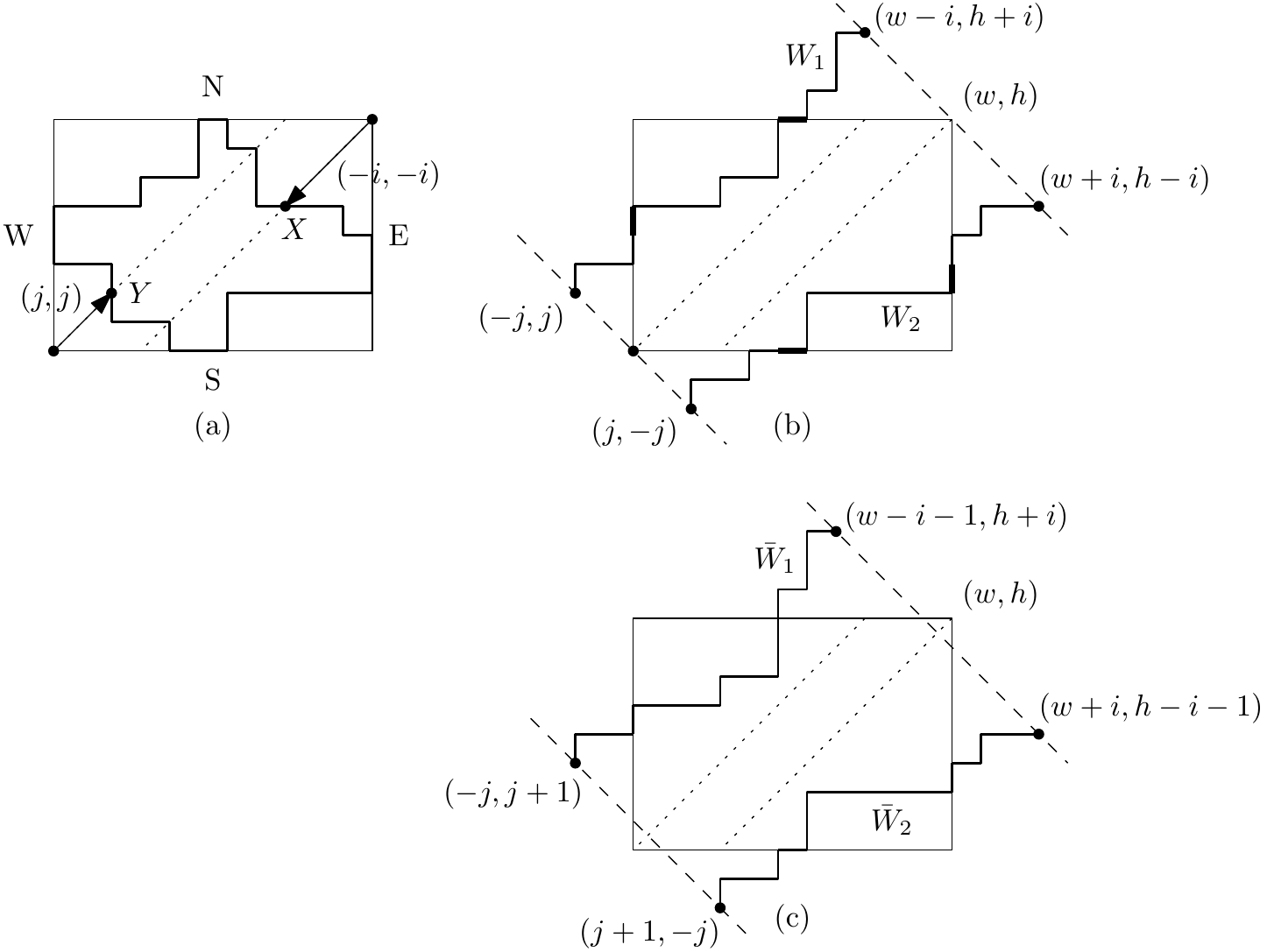}
  \caption{Splitting the boundary at two diagonals (a) and reflecting
    at each side. (b)~An implicit edge on each side of $R$ is
    highlighted. (c) The four implicit edges are removed.}
  \label{fig:split2}
\end{figure}
The two relevant cases are  SWNE (the ``normal case'') and SWEN (the
``flipped case'').
This time, we split the curve not only
at the upper right part but also at the lower left part,
see Figure~\ref{fig:split2}.
More precisely, in addition to splitting the S-walk
at the intersection point $X$
of the falling path between
the N-side and E-side with the diagonal through the NE corner, we also split it
at the intersection point $Y$
of the falling path between
the S-side and W-side with the diagonal through the SW corner.
Both in the normal case SWNE, which is illustrated in Figure~\ref{fig:split2},
and in the flipped case SWEN, there is a falling path between S and W and another falling path between N and E, so that the intersection points $X$ and $Y$ are well-defined.

The two resulting paths intersect if and only if the two rising
paths of the S-walk intersect. Thus we are interested in
the number of intersecting pairs of paths.
In the normal case SWNE, this is
\begin{align}
 \sum_{i\ge 0} \sum_{j\ge 0}
           X_{(j+1,-j),(-j,j+1)}^{(w+i,h-i-1),(w-i-1,h+i)}
= \sum_{i\ge 0} \sum_{j\ge 0}
\label{cross-2}
           N_{(-j,j+1)}^{(w+i,h-i-1)}
           N_{(j+1,-j)}^{(w-i-1,h+i)}.
\end{align}
For the flipped case SWEN, we get a pair of
paths
from
 $(-j,j+1)$ to
           $(w+i,h-i-1)$
and from
           $(j+1,-j)$ to $(w-i-1,h+i)$.
 All these path pairs are forced to intersect
by the positions of their endpoints,
 and their number equals
the same expression~\eqref{cross-2}. Consequently,
\begin{equation} \nonumber 
  P_{wh}^*
 = 2
    \sum_{i\ge0}
    \sum_{j\ge0}
           N_{(-j,j+1)}^{(w+i,h-i-1)}
           N_{(j+1,-j)}^{(w-i-1,h+i)}
= 2
    \sum_{i\ge0}
    \sum_{j\ge0}
    \binom{\s-2}{w+i+j}
    \binom{\s-2}{w-2-i-j}.
\end{equation}
We note for later use that the number of
S-walks of the normal type SWNE
in which the {rising} paths intersect
is $P_{wh}^*/2$.

Since $i$ and $j$ always occur together in the above summation for
$P_{wh}^*$, 
we group terms with the same sum $i+j$:
\begin{align*}
  P_{wh}^*
  &= 2
    \sum_{i\ge0}
    \sum_{j\ge0}
    \binom{\s-2}{w+i+j}
    \binom{\s-2}{w-2-i-j}
    \\
  &= 2
    \sum_{k\ge0}
    \sum_{\substack{i+j=k\\i,j\ge0}}
    \binom{\s-2}{w+i+j}
    \binom{\s-2}{w-2-i-j}
= 2
    \sum_{k\ge0}
    (k+1)
    \binom{\s-2}{w+k}
    \binom{\s-2}{w-2-k}.
\end{align*}
In order to get rid of the factor $(k+1)$ we split it into
two terms
$
(w+k)-(w-1)$, which can be more readily incorporated into the
binomial coefficients:
\begin{align*}
  P_{wh}^*
  &=
    2
    \sum_{k\ge0}
    (w+k)
    \binom{\s-2}{w+k}
    \binom{\s-2}{w-2-k}
-   2
    \sum_{k\ge0}
    (w-1)
    \binom{\s-2}{w+k}
    \binom{\s-2}{w-2-k}
    \\
  &=
    2
    \sum_{k\ge0}
    (\s-2)
    \binom{\s-3}{w+k-1}
    \binom{\s-2}{w-2-k}
-   2
    \sum_{k\ge0}
    (w-1)
    \binom{\s-2}{w+k}
    \binom{\s-2}{w-2-k}
    \\&= 2(\s-2)T_1 - 2(w-1)T_2.
\end{align*}
We will analyze the two sums $T_1$ and $T_2$ in the last expression separately.
For the first sum 
\begin{align*}
T_1=      \sum_{k\ge0}
    \binom{\s-3}{w+k-1}
    \binom{\s-2}{w-2-k},
\end{align*}
we have the following combinatorial interpretation, see
Figure~\ref{model}.
\begin{figure}[tbh]
  \centering
  \includegraphics{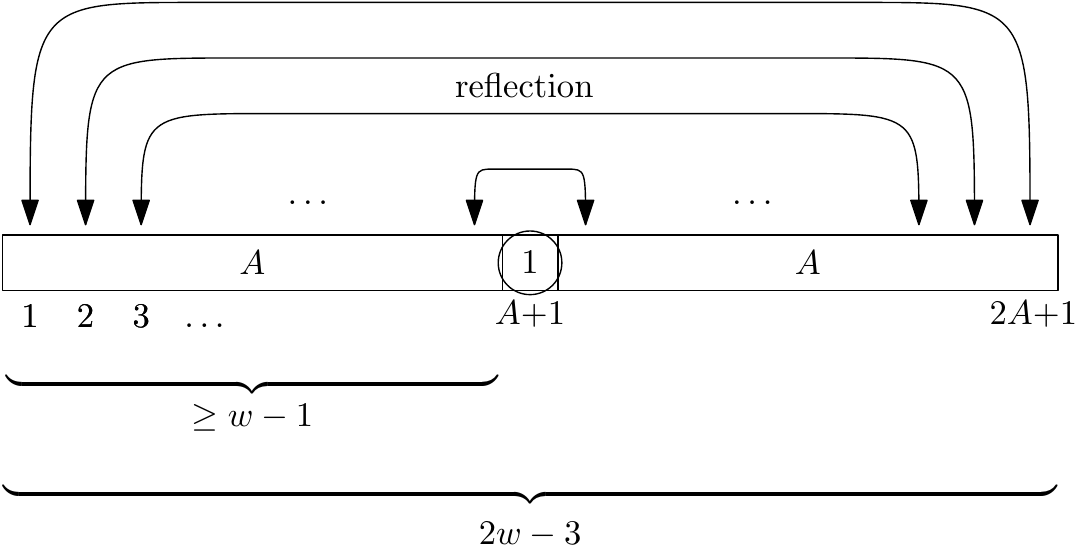}
  \caption{The combinatorial model for the expression $T_1$.}
  \label{model}
\end{figure}
Let $A=\s-3$.
From the numbers $1,2,\ldots$ up to $(\s-3)+(\s-2)=2\s-5=2A+1$, we have to select
a subset $S$ of $(w+k-1)+(w-2-k)=2w-3$ elements, under the additional constraint that
the number of elements selected out of the first $A=\s-3$ integers is at
least $w-1$. (More precisely, it is $w-1+k$.)
$T_1$ counts the number of these subsets.

We consider the subsets $S$ 
under the
reflection operation $S\mapsto \bar S$ which maps each integer $x$ to
its ``mirror image'' $
2A+2-x$. Obviously, for any pair $(S,\bar S)$,
at most one of the two sets can satisfy
the additional constraint. If $S$ contains at least $w-1$ elements in the first $A$ integers, then $S$ contains at most $2w-3-(w-1) = w-2$ in the last $A$ integers, so
$\bar{S}$ contains at most $w-2$ elements in the first $A$ integers. If both $S$ and $\bar{S}$ contain at most $w-2$ elements in the first $A$ integers, we have the \emph{exceptional pairs} $(S,\bar S)$, which are
characterized as
those pairs where
$S$ contains
exactly $w-2$ elements from the first $A$ integers,
exactly $w-2$ elements from the last $A$ integers,
and also the central element $A+1$.

Thus, to count the subsets that
 fulfill the additional constraint, we take all subsets of size $2w-3$,
 subtract the 
 exceptional pairs, and divide by~2.
 \begin{align*}
   T_1=\frac12\left[\binom{2\s-5}{2w-3} - \binom{\s-3}{w-2}^2\right]
 \end{align*}

 For the second sum
 \begin{align}
   \label{S2-1}
T_2=      \sum_{k\ge0}
    \binom{\s-2}{w+k}
    \binom{\s-2}{w-2-k},
\end{align}
an argument similar to $T_1$ is possible, but it is easier to argue
algebraically:
 Substituting $k=2-k'$ and renaming $k'$ to $k$ again gives
a summation of the same expression over a nearly complementary range of $k$.
\begin{align}
  \label{S2-2}
  T_2
  &=      \sum_{k'\le -2}
    \binom{\s-2}{w-2-k'}
    \binom{\s-2}{w+k'}
  =      \sum_{k\le -2}
    \binom{\s-2}{w+k}
    \binom{\s-2}{w-2-k}
\end{align}
Summing up \eqref{S2-1}
and \eqref{S2-2} gives almost the full range of integers $k$.
We add the term for $k=-1$ 
and get
\begin{align*}
  \sum_{k\le -2}+
  \sum_{k= -1}+
  \sum_{k\ge 0}
&  =      \sum_{k-\infty}^\infty
    \binom{\s-2}{w+k}
    \binom{\s-2}{w-2-k}
=    \binom{2\s-4}{2w-2}.
\\
\intertext{On the other hand, from
\thetag{\ref{S2-1}--\ref{S2-2}},}
  \sum_{k\le -2}+
  \sum_{k= -1}+
  \sum_{k\ge 0}
&  =
  T_2+  \binom{\s-2}{w-1}^2 + T_2,
  \end{align*}
and thus
 \begin{align*}
   T_2=\frac12\left[\binom{2\s-4}{2w-2} - \binom{\s-2}{w-1}^2\right].
 \end{align*}

 On the whole we get
\begin{align*}
  P_{wh}^*
  &= (\s-2)\left[\binom{2\s-5}{2w-3} - \binom{\s-3}{w-2}^2\right]
     - (w-1)\left[\binom{2\s-4}{2w-2} - \binom{\s-2}{w-1}^2\right]
\\  &= -(\s-2)\times \binom{\s-3}{w-2}^2
      + (w-1)\times \binom{\s-2}{w-1}^2,
\end{align*}
because the two leading terms cancel.
Taking out common factors, we get
\begin{align}\nonumber
  P_{wh}^*
  &=\left(
    \frac{ (\s-3)!}{(w-1)!(h-1)!}\right)^2
    \left[-(\s-2)\times (w-1)^2
    + (w-1)\times (\s-2)^2
    \right]
  \\\nonumber
  &=\left(
    \frac{ (\s-3)!}{(w-1)!(h-1)!}\right)^2
    (\s-2) (w-1)
    \left[- (w-1)
    +  (\s-2)
    \right]
  \\
  &=\left(
    \frac{ (\s-3)!}{(w-1)!(h-1)!}\right)^2
(\s-2)(w-1)(h-1)
           = (\s-3)\binom{\s-2}{w-1}\binom{\s-4}{w-2}.
\label{Pwh*}
\end{align}
Substituting 
 \eqref{Pwh*}
 and \eqref{s-walks}
 into  \eqref{all-convex} proves
Theorem~\ref{th:all-convex}. \qed

\subsection{Our proof and its relation to the proof of Guo and Zeng}
\label{sec:GZ}
Some of the ideas that we use are implicit in
the proof of
 Guo and
 Zeng~\cite{GZ}, of which we were not aware when we developed this work.
Their derivation 
makes heavy use of Lemma~\ref{lemma-intersecting}, but the
interpretation of the terms as crossing paths is soon lost in the algebraic manipulations.
We will review their approach, using our notation and terminology.

 The number of convex polyominoes are represented as a sum of five expressions:
$$P_{wh}=S_0+S_1+S_2-S_3-S_4$$
The first term,
$S_0=B_{wh}$, represents the number of parallelogram polyominoes, which
contain both the SW and the NE corners
(Proposition~\ref{proposition-parallelogram}).
$S_1$ represents the number of directed polyominoes and the
``opposite directed polyominoes'' that
contain the NE corner, but exclude the parallelogram polyominoes.
Thus, $S_1=2(D_{wh}-B_{wh})$, 
and it is possible to recover the formula for
$D_{wh}$
(Theorem~\ref{directed}) from the expressions for $S_0$ and
$S_1$ in~\cite[Eqs.~\thetag4 and \thetag{5}]{GZ}.
Like in our proof, the calculations leading to
$S_1$ involve a telescoping sum.

The number of the {remaining} polyominoes are represented by
$S_2-S_3-S_4$. Here
$S_2$ counts S-walks that visit the sides in the normal SWNE order,
with the exclusion of the SW and NE corners.
From this, one has to subtract
the number
$S_3$ of those walks where the falling pieces intersect,
and the number
$S_4$ of walks where the rising pieces intersect.
In our notation, we can express this as
$S_3=P^*_{wh}/2$, and indeed,
formula~\eqref{Pwh*} is in accordance with the expression for
$S_3$ in~\cite[Eq.~\thetag9]{GZ}.
$S_4$ has to be computed separately because
 the special treatment of the SW and NE corner has destroyed the
symmetry between ascending and descending intersections.
As mentioned in the introduction, the computation of $S_3$ and $S_4$ involves the
use of generating functions.

By contrast, in our proof, we have tried to reduce the amount of
computation and to give bijective proofs whenever possible.
Some ideas, like cutting at diagonals and reflecting along the
boundary,
were helpful in this respect.
Another distinguishing feature is that we include S-walks that
touch the boundary in an arbitrary order.
This has allowed us to keep the proof simple and elementary.
We did not have to resort to generating functions, and we needed only moderate
algebraic manipulations.

S-walks that touch the boundary in an arbitrary order also show up in
the rejection-sampling procedure of Hochstättler, Loebl and
Moll~\cite{HLM} for convex polyominoes of given perimeter, see
Section~\ref{perimeter}, but their connection to the positive terms in
\eqref{Pm} was not noted.

\section{Moments of Intersection Numbers Between Two Paths}
\label{moments}
In an early attempt to establish Theorem~\ref{th:all-convex}
we tried a different approach to treat intersecting paths:
In contrast to Lemma~\ref{lemma-intersecting}, which deals with intersections
by swapping the endpoints,
 we split the walk into an initial part
up to the last intersection point, plus a final part.
The final part is a directed convex polyomino, 
for which we know the formula.
This approach
 leads to a sum
that can be interpreted as
the total number of intersections of all
pairs of monotone paths between two points.
We failed to evaluate this sum directly.
However, as shown in Theorem~\ref{th:all-convex},
we found a different way to calculate the result.
Turning the argument around, we can therefore say something about
the total number of intersections of all
path pairs (the unnormalized first moment). We can even extend the method and compute
the sum of squares of these intersection numbers (the unnormalized
second moment).

 \begin{theorem}
Let $w,h\ge 0$, and set $\s=w+h$.
Let $\Paths_{wh}$
denote the set of the $\binom{\s}w$ monotone grid paths
from $(0,0)$ to $(w,h)$.
\begin{align}\label{first}
&    \sum_{U\in \Paths_{wh}}
    \sum_{V\in \Paths_{wh}}
    {\lvert U\cap V\rvert}
=  \binom{2\s+2}{2w+1}/2
\qquad\textup{\oeis{A091044}}
\\\label{second}
 &   \sum_{U\in \Paths_{wh}}
    \sum_{V\in \Paths_{wh}}
    \lvert U\cap V\rvert^2
=
(\s+1)\binom{\s+2}{w+1}\binom{\s}{w}
  -\binom{2\s+2}{2w+1}/2
\qquad  \textup{\oeis{A324010}}
\end{align}
\end{theorem}
Dividing these formulas by the number $\binom \s w^2$ of path pairs yields
the (normalized) first moment, i.\,e., the mean,
and the 
second moment, from which
the variance can be calculated. 
The
intersection points of $U$ and $V$ include the
common endpoints, and therefore, $\lvert U\cap V\rvert$ is
always at least 2 (unless $w=h=0$).

%



\begin{proof}
We will establish a bijection
between the left side of~\eqref{first} and \emph{weak directed S-walks} in a
$(w+1)\times(h+1)$ rectangle,
see Figure~\ref{fig:split-off}. Such a
walk starts and ends at $(0,0)$ but it need not contain an edge of
the S- or W-side. The edges on the N- and E-side are, however, required.
Using the code in Section~\ref{sec:number-S-walks}, such a walk can be
coded with $2w+1$ V's and $2h+1$ H's. We orient the walk by requiring
that the N-side is touched before the E-side.

 The number $ D^-_{wh}$ of these walks is
\begin{displaymath}
     D^-_{wh}
  =
               \binom{2\s+2}{2w+1}/2.
\end{displaymath}
The factor $1/2$ accounts for the orientation.
\begin{figure}[tbh]
  \centering
  \includegraphics{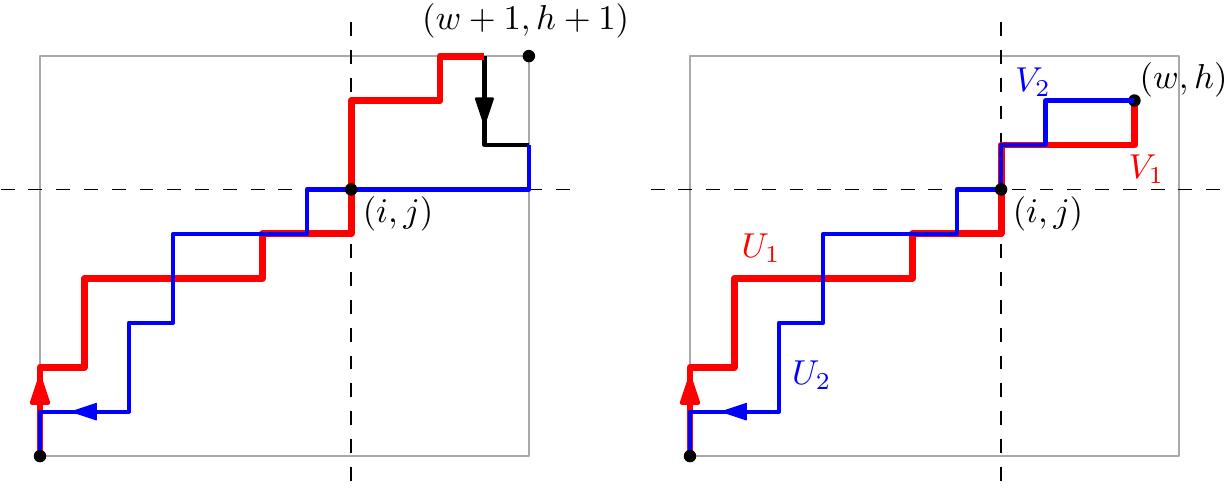}
  \caption{
A weak directed S-walk is
split at the highest intersection point and converted into a pair of paths.}
  \label{fig:split-off}
\end{figure}

In a weak directed S-walk, a self-intersection can only occur between the two
rising pieces.
Now we split the walk at the highest and rightmost
intersection point $(i,j)$. If there is no intersection point, we
set $(i,j)=(0,0)$. Now the part $P_2$ that was split off is
a directed convex polyomino, spanning the rectangle
$[i,w+1]\times[j,h+1]$,
and containing the SW corner $(i,j)$ of this rectangle.
We have established
through Theorem~\ref{directed}
that such directed convex polyominoes are in
bijection with ordered pairs $(V_1,V_2)$ of monotone paths
from $(i,j)$ to $(w,h)$.
After
splitting off $P_2$ from the S-walk,
an initial path $U_1$ from $(0,0)$ to $(i,j)$
and a final path $U_2$ from $(i,j)$ to $(0,0)$ remain.
We connect
$U_1$
with $V_1$ and the reverse of $U_2$
with $V_2$, obtaining
 a pair of paths from $(0,0)$ to
$(w,h)$. We mark the intersection point $(i,j)$.

Conversely, consider a pair of paths
 from $(0,0)$ to
$(w,h)$, in which one of the intersection points $(i,j)$ is marked.
We split off the parts after $(i,j)$ and turn them into a directed
polyomino, using the bijection of
 Theorem~\ref{directed}. We obtain a directed S-walk.

Thus we have established a bijection between
weak directed S-walks and
pairs of paths $W_1,W_2\in \Paths_{wh}$ in which an
intersection in $W_1\cap W_2$ is marked.
The number of objects of the latter type is the number of path pairs,
weighted by the number of intersections; this is the quantity that is
counted in~\eqref{first}.

We note that both the first point $(0,0)$ and the last
point
$(w,h)$ are valid intersection points to be marked.
A mark on the first point corresponds to
 directed S-walks without self-intersections.
If the last point is marked, the bijection will expand it into a path around a unit square.

To get the second moment~\eqref{second}, we perform the replacement on both ends.
We start with a pair of paths $(U,V)$ in which \emph{two} indistinguishable marks are
placed on intersection points. The two marks can be placed on the
same point.
The number of these objects is
\begin{displaymath}
      \sum_{U\in \Paths_{wh}}
    \sum_{V\in \Paths_{wh}}
    \binom{\lvert U\cap V\rvert+1}2.
\end{displaymath}
 The pair of paths after the higher mark is replaced by a
 directed
 polyomino as above.
 The pair of paths up to the lower mark is replaced by a
another
 directed
 polyomino, which is rotated by $180^\circ$: instead of the lower left
 corner, it contains the upper right corner
 of its enclosing rectangle.
We obtain an S-walk
 of the normal type SWNE
 in a
$(w+2)\times(h+2)$ rectangle,
 whose
rising parts
intersect.
As argued in Section~\ref{self-intersecting} this number is
$P^*_{w+2,h+2}/2$, the other half being the flipped type SWEN.
By~\eqref{Pwh*}, this gives
\begin{equation}\label{second-binomial}
      \sum_{U\in \Paths_{wh}}
    \sum_{V\in \Paths_{wh}}
    \binom{\lvert U\cap V\rvert+1}2
       = (\s+1)\binom{\s+2}{w+1}\binom{\s}{w}/2.
\end{equation}
Since $\binom {k+1}2 = (k^2+k)/2$,
with $k=\lvert U\cap V\rvert$,
 the claimed formula
\eqref{second}
 for the second
moment
is obtained by multiplying
\eqref{second-binomial} by 2 and subtracting
\eqref{first}.
\end{proof}

\section{Rejection Sampling for Convex Polyominoes}
\label{sec:sampling}

Since the formula~\eqref{P} contains a subtraction,
it does not suggest a way to bijectively sample
 convex polyominoes with given height and width.
However, it is straightforward to randomly generate one of the
$  \tilde P_{wh}
  =
               \binom{2\s-4}{2w-2} + (w-1)\binom{2\s-5}{2w-2}$
S-walks. Note that they correspond to the positive terms in
formula~\eqref{P}. If the S-walk intersects itself, it is
rejected, and another random S-walk must be tried.
The \emph{efficiency} of the method is the success probability of generating
a valid polyomino.

For
the square case $w=h$
the probability of self-intersection goes to 0 like
$4\sqrt{2/\pi w}$, and hence the efficiency goes to 1.
In general, the efficiency is $1-O(1/\sqrt{\min\{w,h\}})$.
For small values of $h$ or $w$, the efficiency is around $1/2$.
The worst observed efficiency is around
$48\%$, for $h=w=4$.

\subsection{Sampling for given perimeter $2\s$}
\label{perimeter}
For sampling from the set of polyominoes with a given perimeter $2\s$,
without regarding the specific width and height, we can give a more explicit
sampling procedure with an efficiency that approaches 1 as $\s$
increases.
A very similar procedure is proposed
in
Hochstättler, Loebl and Moll~\cite{HLM}: its efficiency approaches
$1/2$ for large~$\s$. By adding a small twist to their procedure, we
can avoid throwing away $50\,\%$ of the samples.

The procedure is the same as above: We generate all S-walks.
The sum
\begin{equation*}
\tilde
  P_\s
=\sum_{w+h=\s} \tilde P_{wh}
  =
      4^{\s-2}/2+
      4^{\s-3}/2\cdot (2\s-5)/2
  =
      4^{\s-4}(2\s+3)
\end{equation*}
has a term
$4^{\s-2}/2$
for the directed polyominoes ($a=0$) and another term
$4^{\s-3}/2\cdot (2\s-5)/2$ for the remaining polyominoes ($a\ge 1$).

The S-walks for the directed polyominoes are obtained from
strings of
length $2\s-4$
 over the alphabet $\{\mathrm{V},\mathrm{H}\}$
 with an even number of H's. Their number is
 $(1/2)2^{2\s-4} = 4^{\s-2}/2$.

Let us now concentrate on the case $a\ge 1$.
We assume $\s\ge 4$.
We start by generating a random string of
length $2\s-6$
 over the alphabet $\{\mathrm{V},\mathrm{H}\}$
with an odd number of H's
($2^{2\s-7}$ possibilities).
We generate a random position where to insert an additional H ($2\s-5$
possibilities).
The result is a random string of length
$2\s-5$ with some even number
$\#\textrm{H}\ge 2$ of H's
and an odd number
$\#\textrm{V}\ge1$ of V's.
We choose $w$ and $h$ such that
$\#\textrm{H}=2w-2$ and
$\#\textrm{V}=2h-3$.

 Suppose that the additional, inserted H
is the $\bar a$-th H from the left.
This gives a random number $\bar a$
in the range $1\le \bar a \le 2w-2$.
We reduce this number modulo $w-1$ so that the resulting value $a$ lies
 in the interval $1\le a \le w-1$.
This gives all the data we need for an S-walk of height $h$ and width
$w$ with
 $1\le a \le w-1$.

We have started with $2^{2\s-7}(2\s-5)$
 possibilities, and we have used them to generate one of
$4^{\s-3}/2\cdot (2\s-5)/2
=2^{2\s-8}\cdot (2\s-5)$ random objects.
Thus we have wasted
 one bit. This happened at the reduction from $\bar a$ to $a$
modulo $w-1$.
If this loss is an issue, the lost bit can be recovered
 from the reduction (by comparing $\bar a$ against $w-1$)
and recycled for the next round of random generation.
 We leave it as a challenge to 
avoid the generation of the extra bit from the start.

The overall procedure,
 including the case $a=0$,
 is as follows.
We generate
a random string $x$ of
length $2\s-7$
 over the alphabet $\{\mathrm{V},\mathrm{H}\}$
and a random number $Q$ in the range $1\le Q\le 2\s+3$.
If $Q> 2\s-5$, we translate the eight possibilities into 3 bits and
use them to extend $x$ by 2 more symbols. One bit is wasted.
We add a parity symbol so that we get a string
of length
$2\s-4$ with an even number of H's and V's.
This is used as an S-walk with $a=0$.

If $Q\le 2\s-5$, we are in the case $a\ge 1$.
We extend $x$ by a parity symbol to make the number of H's odd, and we
use $Q$ as the insertion point for the additional $H$, as described
above.

The efficiency of the method is easy to analyze.
By comparing the negative term of formula~\eqref{Pm} against the
positive term, one can see
the probability of self-intersection is goes to 0 like
$4/\sqrt{\pi\s}$ as $\s\to\infty$.
(If we don't recycle the wasted bit, we additionally loose 1 bit
out of $2\s-7+\log_2(2\s+3)$ bits.)
 Hence the efficiency approaches~1 as $\s$ increases.

The main difference in the procedure of
Hochstättler, Loebl and Moll~\cite{HLM} is that they generate
the code string and the parameter $a\in\{0,1,2,\ldots,\s-1\}$
independently.
If the number of H's is bigger than $2a$, they certainly have to reject the
sample.
This causes them to loose 50\,\% of the samples right away.
Our improvement is essentially due to a combinatorial interpretation
of the equality in~\eqref{s-walks}, followed by a summation over all
pairs $w,h$ with $w+h=\s$.

\end{document}